\documentclass[letterpaper,10pt,reqno]{amsart}
\usepackage{indentfirst} % Indent first paragraph after section header
\usepackage{amssymb}
\usepackage{mathtools} % improves amsmath
\usepackage{mathabx} % makes many symbols (as $\leq$) more beautiful
\usepackage{amsthm}
\usepackage{thmtools}
\usepackage{enumitem} % special itemize with custom tags and labels that work
\usepackage[colorlinks=true]{hyperref} % Links in the pdf. Backref option: switched off
\usepackage[usenames,dvipsnames]{xcolor} % With XCOLOR, printed quality is good, (but with COLOR it's bad)
\usepackage{ifthen}
\usepackage{tikz}
\usetikzlibrary{decorations.pathreplacing}
\usepackage{tikz-cd}
  
\makeatletter

\def\MRbibitem{\@ifnextchar[\my@lbibitem\my@bibitem}

\def\mybiblabel#1#2{\@biblabel{{\hyperref{http://www.ams.org/mathscinet-getitem?mr=#1}{}{}{#2}}}}

\def\myhyperanchor#1{\Hy@raisedlink{\hyper@anchorstart{cite.#1}\hyper@anchorend}}

\def\my@lbibitem[#1]#2#3#4\par{%
  \item[\mybiblabel{#2}{#1}\myhyperanchor{#3}\hfill]#4%
  \@ifundefined{ifbackrefparscan}{}{\BR@backref{#3}}%
  \if@filesw{\let\protect\noexpand\immediate% write to aux-file
    \write\@auxout{\string\bibcite{#3}{#1}}}\fi\ignorespaces%
}

\def\my@bibitem#1#2#3\par{%
  \refstepcounter\@listctr% standard tex item counter for the generic item number
  \item[\mybiblabel{#1}{\the\value\@listctr}\myhyperanchor{#2}\hfill]#3%
  \@ifundefined{ifbackrefparscan}{}{\BR@backref{#2}}%
  \if@filesw\immediate\write\@auxout% write to aux-file
    {\string\bibcite{#2}{\the\value\@listctr}}\fi\ignorespaces%
}

\makeatother
\DeclareFontFamily{U} {MnSymbolA}{}
\DeclareFontShape{U}{MnSymbolA}{m}{n}{
   <-6> MnSymbolA5
   <6-7> MnSymbolA6
   <7-8> MnSymbolA7
   <8-9> MnSymbolA8
   <9-10> MnSymbolA9
   <10-12> MnSymbolA10
   <12-> MnSymbolA12}{}
\DeclareFontShape{U}{MnSymbolA}{b}{n}{
   <-6> MnSymbolA-Bold5
   <6-7> MnSymbolA-Bold6
   <7-8> MnSymbolA-Bold7
   <8-9> MnSymbolA-Bold8
   <9-10> MnSymbolA-Bold9
   <10-12> MnSymbolA-Bold10
   <12-> MnSymbolA-Bold12}{}
\DeclareSymbolFont{MnSyA} {U} {MnSymbolA}{m}{n}
 \DeclareFontFamily{U} {MnSymbolC}{}
\DeclareFontShape{U}{MnSymbolC}{m}{n}{
  <-6> MnSymbolC5
  <6-7> MnSymbolC6
  <7-8> MnSymbolC7
  <8-9> MnSymbolC8
  <9-10> MnSymbolC9
  <10-12> MnSymbolC10
  <12-> MnSymbolC12}{}
\DeclareFontShape{U}{MnSymbolC}{b}{n}{
  <-6> MnSymbolC-Bold5
  <6-7> MnSymbolC-Bold6
  <7-8> MnSymbolC-Bold7
  <8-9> MnSymbolC-Bold8
  <9-10> MnSymbolC-Bold9
  <10-12> MnSymbolC-Bold10
  <12-> MnSymbolC-Bold12}{}
\DeclareSymbolFont{MnSyC} {U} {MnSymbolC}{m}{n}

\DeclareMathSymbol{\top}{\mathord}{MnSyA}{219} % smaller symbol for transpose
\DeclareMathSymbol{\plus}{\mathord}{MnSyC}{20} % a smaller plus sign

\declaretheorem[numberwithin=section]{theorem}
\declaretheorem[sibling=theorem]{lemma}
\declaretheorem[sibling=theorem]{corollary}
\declaretheorem[sibling=theorem]{proposition}
\declaretheorem[sibling=theorem,style=definition]{definition}

\declaretheorem[sibling=theorem,style=remark]{remark}

\hypersetup{bookmarksdepth = 3} % Depth of sections/subs... to have bookmark links in the pdf;
\numberwithin{equation}{section}     % Makes labeled equations easier to find.

\setlist[enumerate,1]{label={\upshape(\alph*)},ref=\alph*}
\setlist[enumerate,2]{label={\upshape(\arabic*)},ref=\arabic*}

\newcommand{\R}{\mathbb{R}}
\newcommand{\Z}{\mathbb{Z}}
\newcommand{\N}{\mathbb{N}}

\def\phi{\varphi}
\def\R{{\mathbb R}}

\def\N{{\mathbb N}}
\def\Z{{\mathbb Z}}

\def\Q{{\mathbb Q}}

\newcommand{\vertiii}[1]{{\left\vert\kern-0.25ex\left\vert\kern-0.25ex\left\vert #1 
    \right\vert\kern-0.25ex\right\vert\kern-0.25ex\right\vert}}
\newcommand{\invertiii}[1]{{\vert\kern-0.25ex\vert\kern-0.25ex\vert #1 
    \vert\kern-0.25ex\vert\kern-0.25ex\vert}}

\begin{document}

\title{Odometers, Backward continued fractions and counting rationals}
%\date{\today}



\begin{thanks}  {We would like to thank An\'ibal Velozo for interesting discussions on the topic of this article and to Giovanni Panti for informing us of \cite{bi}. We would also like to thank the referees and the editor for their comments and suggestions. Iommi was partially supported by Proyecto Fondecyt 1230100. Ponce was partially supported by Proyecto Fondecyt 1220032.}
 \end{thanks}

\author[G.~Iommi]{Godofredo Iommi} \address{Facultad de Matem\'aticas,
Pontificia Universidad Cat\'olica de Chile (UC), Avenida Vicu\~na Mackenna 4860, Santiago, Chile}
 \email{\href{mailto:godofredo.iommi@gmail.com}{godofredo.iommi@gmail.com}} 
\urladdr{\href{http://www.mat.uc.cl/~giommi}{www.mat.uc.cl/$\sim$giommi}}

 \author[M.~Ponce]{Mario Ponce}   \address{Facultad de Matem\'aticas,
Pontificia Universidad Cat\'olica de Chile (UC), Avenida Vicu\~na Mackenna 4860, Santiago, Chile}
\email{\href{mponcea@mat.uc.cl}{mponcea@mat.uc.cl}}

\maketitle

\begin{abstract}
It has been more than twenty years since Moshe Newman, based on work by Neil Calkin and Herbert Wilf, introduced  an \emph{explicit} bijection between the rational and natural numbers. Interestingly, this bijection is dynamic in nature. Indeed, Newman's map  has the property that the orbit of zero provides the required bijection. Claudio Bonanno and Stefano Isola, using continued fractions expansions, described the dynamics of  its first return time map $T$. They proved that it is topologically conjugated to the dyadic odometer. In this article, we prove that the \emph{correct} numerical system needed to analyze this map is the backward continued fractions. Indeed, this approach has the advantage that it provides  \emph{explicitly} the action of $T$ on the expansion. As a by-product, we naturally obtain an \emph{explicit} formula for Minkowski's  question mark function in terms of backward continued fractions.  The whole point of Newman was to provide an explicit bijection, our approach shares the same taste for the explicit.
\end{abstract}

%
%We resume the study of an interval map  $T$, that corresponds to the first return map of the {\it Newman--Calkin--Wilf} function. It has the property that the orbit of zero induces an \emph{explicit} bijection between the rational numbers in the unit interval and the natural numbers. Bonanno and Isola, using continued fractions expansions, proved that $T$ is topologically conjugated to the dyadic odometer. We propose a proof of the same result that, instead of continued fractions, uses backward continued fractions. This approach has the advantage that it provides  \emph{explicitly} the action of $T$ on the expansion. As a by product, we naturally obtain an \emph{explicit} formula for Minkowski's  question mark function in terms of backward continued fractions. 

\section{Introduction.}

We will essentially tell a \textsc{Monthly}  tale. Interestingly, in Spanish the word for  {\it counting} and that for {\it telling} is actually the same, \emph{contar}. In 1873, Cantor established that the rational numbers are countable. This is a remarkable result not only because the natural numbers are strictly contained in the rationals, but because the later are dense in the (uncountable) set of real numbers. Most proofs of this fact do not provide an explicit bijection between $\N$ and $\Q$. In this article we will be interested in an \emph{explicit} bijection that was defined and studied at length in this journal. \\

The year 2000, in their \textsc{Monthly}  article \cite{cw}, Calkin and Wilf constructed a sequence of positive integers, $b(n)$, with the property that every positive rational appears only once as $b(n)/b(n+1)$. The numbers $b(n)/b(n+1)$ were easily arranged in a dyadic tree. Moreover,  simple rules to obtain the two descendants were given. This tree is now, sometimes, {{}{referred}} to as the Calkin and Wilf tree. In 2003,  answering a question of Donald Knuth (\textsc{Monthly}  Problem 10906), and inspired by the work of Calkin and Wilf,  Newman \cite{k} constructed a dynamical system with the property that the orbit of zero induces such a bijection. This means that he defined a function $F:[0,\infty) \to [0, \infty)$ with the property that the map $n \mapsto F^n(0)$, where $F^n(x)= (F \circ \dots \circ F)(x)$ denotes the $n-$th iterate of $F$, is the required bijection.\\

We will study the first return map of the map $F$ to the unit interval, which we denote by $T$. It turns out that $T$ is nothing but the second iterate of $F$. Moreover, the orbit of zero with respect to $T$ induces a bijection between the natural numbers and $\Q \cap [0,1)$. Bonnano and Isola \cite{bi}, {{}{in}} the year 2009, studied this map in detail. They proved that, from a dynamical systems point of view, the map $T$ is equivalent (topologically conjugated) to a well known system called \emph{dyadic odometer}. This later system was apparently introduced by John von Neumann and also studied by Shizuo 
 Kakutani. All of its orbits are dense and it has a unique invariant probability measure. It acts in a similar way as car devices that count the kilometers driven, hence its name. The proof of Bonnano and Isola uses the continued fraction expansion of real numbers. \\

In this note we propose a proof of the result by Bonnano and Isola that instead of using continued fractions makes use of \emph{backward continued fractions}, see Theorem \ref{main}.  We believe that this numeration system, studied in detail at least since the 1957 work of Alfred Renyi \cite{re},  is better suited to study the map $T$. Indeed, the motivation to  study the map $T$ is that it provides an \emph{explicit} bijection between natural numbers and rationals in the unit interval. Our approach, by means of backward continued fractions, provides an \emph{explicit} action of $T$ on this numeration system (see Theorem \ref{action}).  We hope to satisfy readers with a taste for the explicit.\\

Roughly speaking, the arguments by Bonanno and Isola relating the odometer with continued fractions expansions, runs through the following lines of thought. The Minkowski Question Mark function, denoted by $
?$, conjugates the dynamics of the Tent map with that of the Farey map. Moreover, it turns out that the Gauss map is nothing but the \emph{accelerated} Farey map (or the corresponding \emph{jump transformation}). That is, 
\begin{equation*}
\text{Gauss map} \xleftrightarrow{accelerated} \text{Farey map} \xleftrightarrow{?} \text{Tent map}
\end{equation*}
{ Recall that the Gauss map is the dynamical system associated to the continued fraction expansion. On the other hand,} the Tent map is the dynamical system corresponding to the \emph{alternating binary expansion} (see \cite[Section 2.3.3]{dkr}). This latter expansion associates to each point of the unit interval a sequence of zeroes and ones. The alternating binary expansion can be identified, via the Farey map, with the continued fraction expansion of the point corresponding to $?^{-1}(x)$. In terms of expansions this can be summarized by,
\begin{equation*}
\text{Continued fractions} \xleftrightarrow{accelerated} \text{Farey map} \xleftrightarrow{?} \text{Alternated Binary expansion}
\end{equation*} 
These relations are studied in detail in \cite{kms}, see also \cite[Section 2.1]{i2}. 
This circle of ideas allows for the proof of the fact that $T$ is conjugated, via the Minkowski Question Mark function, to the odometer (as established in \cite{bi}). \\

Nevertheless, the odometer is  defined on the binary expansion of points in $[0,1]$. We propose an alternative approach, in which the binary expansion occurs naturally. The binary expansion is not directly related to the tent map, but to the doubling map. As above, the Minkowski Question Mark function makes a fundamental appearance. In what follows we will develop the following scheme, 
\begin{equation*}
\text{Renyi map} \xleftrightarrow{accelerated} \text{Back. Farey map} \xleftrightarrow{?} \text{Doubling map}
\end{equation*}
{ The Renyi map is the dynamical system associated to the backward continued fraction expansion. Thus, } in terms of expansions this is,
\begin{equation*}
\text{Backward continued fractions} \xleftrightarrow{accelerated} \text{Back. Farey map} \xleftrightarrow{?} \text{Binary expansion}
\end{equation*}

{{}{T}}herefore, the backward continued fraction expansion is the natural way to code if we are to conjugate to an odometer. That is why we obtain a simple and explicit formula for the odometric action of $T$.\\ 

 These two approaches also lead to different expressions for the Minkowski Question Mark function. Indeed, during the 1930s, Arnaud Denjoy \cite{d1,d2} obtained an analytical formula for the Question Mark function in terms of the continued fraction expansion of a number. In this article, we are naturally lead to a similar formula which, instead of continued fractions, describes the Question Mark function in terms of the backward continued fraction expansion, see Proposition \ref{?-b}.

\section{Counting rationals and backward continued fractions.}  

While it is simple to prove that the rational numbers are countable, explicit bijections between $\N$ are $\Q$ are less well known. Newman, solving a problem posed by Knuth \cite{k}, established the following,
\begin{proposition} \label{p}
The orbit of $x=0$ with respect to the map $F:\R^+ \cup\{0\} \to \R^+$, 
\begin{equation*}
F(x)= \frac{1}{2[x] -x+1},
\end{equation*}
where $[x]$ denotes the integer part of $x$, defines a bijection between $\N$ and $\Q^+$. That is, the map $n  \mapsto F^n(0) $ is a bijection between the natural and the positive rational numbers.
\end{proposition}
The proof of this result is based on a construction done by Calkin and Wilf \cite{cw} in which, using  the Euclid tree, they define a sequence $b(n)$ so that the map $n \mapsto \frac{b(n)}{b(n+1)}$ is a bijection between $\N$ and $\Q^+$. A detailed proof of Proposition \ref{p} can be found in  \cite[Chapter 17]{az} and in \cite{msz}.

 \subsection{Backward continued fractions.} There is a wide range of different expansions of real numbers each of which captures different arithmetical, algebraic, or dynamical properties. As it will be established in this article, the \emph{natural} expansion associated to the map $F$ that enumerates the rational numbers is the backward continued fraction. This expansion is related to the Gauss reduction theory of indefinite integer quadratic forms, when considered in matrix terms as $SL_2(\Z)-$equivalent matrices (see \cite{ka,z}). Backward continued fraction expansions of real numbers have been extensively studied from the dynamical point of view at least since the 1957 work of Renyi \cite{re}. 
 
Every real number $x \in [0,1)$ can be written as a \emph{backward continued fraction} of the form
\begin{equation*}
x = 1- \textrm{ } \cfrac{1}{a_1 - \cfrac{1}{a_2 - \cfrac{1}{a_3 - \dots}}} = \textrm{ } [a_1, a_2, a_3, \dots],
\end{equation*}
where $a_i \in \N$ with $a_i \geq 2$. Irrational numbers  have a unique  backward continued fraction expansion which is infinite. Rational numbers, $p/q \in \Q$,  have two different expansions: a finite one $p/q=[a_1, a_2, \dots, a_n]$  and an infinite one of the form $p/q=[a_1, \dots, a_{n}+1, 2,2, \dots ]$ (that ends with a tail of $2$'s), \cite[Theorem 1.2]{ka}. In this article, unless we {{}{state}} the contrary,  we will always consider the infinite expansion of a rational number.  Moreover, any sequence $(a_i)_{i\geq 1}$ with $a_i\geq 2$, corresponds to the backward continued fraction of a unique real number in $[0,1)$. We refer to \cite{bl, ka} for an introduction to backward continued fractions. 

In his seminal article \cite{re}, Renyi  introduced the map $R:[0,1) \to [0,1)$ defined by
\begin{equation}\label{Renyi_form}
R(x)=\frac{1}{1-x} -\left[\frac{1}{1-x} \right].  
\end{equation}
We will refer to $R$ as the \emph{Renyi map}. The coefficients $a_i$ of the backward continued fraction of $x$ can be obtained by the following recursion:
\begin{equation*}
x=x_1,\quad  x_{n+1}=R(x_n),\quad a_n=\left[\frac{1}{1-x_n}\right]+1.
\end{equation*}
The Renyi map acts as the shift on the backward continued fraction (see \cite{ka} and \cite[Chapter 11]{dk}). That is, 
\begin{equation*}
\textrm{ if } x =[a_1, a_2, a_3, \dots] \textrm{ then } R(x)=[a_2, a_3, \dots].
\end{equation*}
The map $R$ also occurs as a Poincar\'e map for certain cross sections of the geodesic flow on the modular surface \cite{af,ka}.  Ergodic properties of this map have been studied, among others, by Roy Adler and Leopold Flatto \cite{af} and by Renyi himself \cite{re}, see also \cite{gh,i1}. This is a map with infinitely many branches and infinite topological entropy. It has a  fixed point at zero, $R(0)=0$, which is parabolic, that is $R'(0)$=1.  Renyi \cite{re} showed that there exists an infinite $\sigma-$finite invariant measure, $\mu_R,$  absolutely continuous with respect to the Lebesgue measure. It is defined by
\[ \mu_R(A) = \int_A \dfrac{1}{x} \ dx, \]
where $A \subset [0,1)$ is a Borel set.  There is no finite  invariant measure absolutely continuous with respect to the Lebesgue measure.  The map $R$, restricted to the irrational numbers is topologically  conjugated to the full shift on a countable alphabet. 

We introduce the following notation $\Sigma_{\geq n}=\{n, n+1, n+2, \dots\}^{\N}$. Endow the set $\N \cup\{0\}$ with the discrete topology and $\Sigma_{\geq n}$ with the product topology. The shift map $\sigma: \Sigma_{\geq n} \mapsto \Sigma_{\geq n}$ is defined by $\sigma(w_1 w_2 \dots)=(w_2 w_3 \dots)$. Denote by $\pi_R:\Sigma_{\geq 2} \mapsto  [0,1)$ the coding map that to each sequence $\omega=(w_1 w_ 2 \dots)$ associates the  real number   $\pi_R(\omega)=[w_1, w_2, \dots]$.  Let $\Sigma_{\bar 2}=\{(w_i)_i \in \Sigma_{\geq 2} : \text{there exist } k \in \N, \text{ such that  }  w_i=2 \text{ for } i \geq k\}$. \begin{lemma} \label{code}
The map $\pi_R:\Sigma_{\geq 2} \mapsto  [0,1)$ satisfies
\begin{equation}
R \circ \pi_R = \pi_R \circ \sigma.
\end{equation}
Moreover,  $\pi_R:\Sigma_{\geq 2}\setminus \Sigma_{\bar 2} \mapsto  [0,1) \setminus \Q$ is a homeomorphism.
\end{lemma}
Recall that sequences in $\Sigma_{\bar 2}$, those having a tail of 2's, represent rational numbers in the backward continued fraction expansion.

 \subsection{First return map of $F$.} In this article we  will be interested in the map $T:[0,1) \to [0,1)$ which corresponds to the first return time map of $F$ to set $[0,1)$.
 
\begin{definition}
Let $T:[0,1) \to [0,1)$ be the map defined by
\begin{equation*}
T(x)= \frac{1}{2\left[\frac{1}{1-x}	\right] + 1 - \frac{1}{1-x}}.
\end{equation*}
\end{definition}
The map $T$ is piecewise continuous in the partition of $[0,1)$ given by $\{n/(n+1) : n\in \N \}$, see Figure \ref{fig}.

\begin{remark}
The following observations hold:
\begin{enumerate}
\item Let $r:[0,1) \to \N $ be the \emph{first return time} of $F$ to $[0,1)$, which is defined by $r(x):= \inf \{ n \in \N : F^{n}(x) \in [0,1) \}$. We have that, $T(x)=F^{r(x)}(x)$.
\item $F(\Q\cap [0,1))=\Q\cap[1,\infty)$ and $F(\Q\cap [1,\infty))=\Q\cap (0,1)$.
\item The previous observation implies that for every $x \in [0,1)$ we have  $r(x)=2$. Thus,  $T(x)=F^2(x)$.
\item It is a consequence of Proposition \ref{p} and the form of $F$ that the map $n \mapsto T^n(0)$ is a bijection between the natural numbers $\N$ and   $\Q \cap (0,1)$.  
\end{enumerate}
\end{remark} 

%The following result provides evidence that, in this setting, the backward continued fraction is the natural way to code points in $[0,1]$. 

 \subsection{The action of $T$ on the backward continued fraction.}

The following results describe the action of the map $T$ on the backward continued fraction of $x \in [0,1)$. Interestingly, it is just the use of backward continued fractions that allows for the proof of such a  neat formula for $T$. In later sections we will prove that an action of the type described here corresponds to that of a simple dynamical system called {{}{the}} \emph{odometer}. We stress, however, that we do not {{}{require}} that to obtain an explicit formula for the action of $T$.

 \begin{lemma}\label{TyR}
 If $x \in [0,1)$ has a backward continued fraction $x=[a_1, a_2, \dots]$ then 
\begin{equation}
T(x)=  \frac{1}{a_1 - R(x)},
\end{equation}
where $R(x)$ is the Renyi map. 
 \end{lemma}
 
 \begin{proof}
 Since $a_1= \left[\frac{1}{1-x}\right]+1$, we have
  \begin{eqnarray*}
 T(x)=\frac{1}{\left[\frac{1}{1-x}	\right] + 1 -\left( \frac{1}{1-x} -\left[\frac{1}{1-x}	\right] \right)}=  \frac{1}{a_1 - R(x)}.
  \end{eqnarray*}
 \end{proof}
%{{} GI: creo que es mejor poner esto aqui, antes de la expansion binaria, del odometro y de la conjugacion...no necestimaos nada de eso para probar este resultado. Eso muestra que nuestro punto de vista es el correcto...no como el de bonano e isola} 

 \begin{theorem} \label{action}
 If $x=[a_1, a_2, a_3, \dots] \in [0,1)$ then
 \begin{equation*}
 T(x)=[\underbrace{2,2,\dots,2}_{(a_1-2)\textrm{-times}}, a_2+ 1, a_3, a_4, \dots].
 \end{equation*}
 \end{theorem}
 
 \begin{proof}
Recall that if $x=[a_1, a_2,  \dots] \in [0,1)$ then
\begin{equation} \label{rep}
x=1- \textrm{ } \cfrac{1}{a_1 - \cfrac{1}{a_2 - \dots}}= 1 - \frac{1}{a_1- 1 + R(x)},
\end{equation}
since the Renyi map acts as the shift on the backward continued fraction.
The proof will be done by induction on the first digit of the expansion.  Let us first consider the case in which $a_1=2$. Due to Lemma \ref{TyR}, we have that
\begin{eqnarray*}
T(x) = \frac{1}{2- R(x)}
= \frac{1}{1 + \frac{1}{a_2-1+R^2(x)}}=
\frac{a_2-1+R^2(x)}{a_2 +R^2(x)}= 1 - \frac{1}{a_2 +R^2(x)},
\end{eqnarray*}
which implies that 
\begin{equation*}
T(x)= [a_2+1, a_3,  \dots].
\end{equation*}
Thus, the case $a_1=2$ is proved. Assume now that for $a_1=n\geq 2$ we have 
\begin{equation*}
T(x)=[\underbrace{2,2,\dots,2}_{(n-2)\textrm{-times}}, a_2+ 1, a_3, a_4, \dots].
\end{equation*}
Consider the case in which 
$a_1=n+1$. We have that
\begin{eqnarray*}
T(x)= \frac{1}{a_1-R(x)}= 1- \frac{1}{\frac{a_1-R(x)}{a_1-R(x)-1}}= 1- \frac{1}{2- \left(1 -\frac{1}{a_1-1-R(x)} \right)}.
\end{eqnarray*}
The inductive assumption on $a_1-1=n$ implies that
\begin{equation*}
\frac{1}{(a_1-1)-R(x)}=[\underbrace{2,2,\dots,2}_{(n-2)\textrm{-times}}, a_2+ 1, a_3, a_4, \dots].
\end{equation*}
Therefore,
\begin{equation*}
T(x)=[ \underbrace{2,2,\dots,2}_{(n-1)\textrm{-times}}, a_2+ 1, a_3, a_4, \dots].
\end{equation*}
This concludes the proof.
\end{proof}

 \section{Binary expansions and the dyadic odometer.}

From the topological point of view dynamical systems are classified by topological conjugacies. That is, two dynamical systems, $A:X \to X$ and $B:Y \to Y$, share the same (topological) dynamical properties, if there exists {{}{a}} homeomorphism $g: X\to Y$ such that $g \circ A \circ g^{-1}= B$. In Section \ref{sec:conj} we will prove that the dynamical system $T$ is topologically conjugated to a very well understood system called {{}{the}} \emph{odometer}.

An odometer  is a device used for measuring the distance traveled by a vehicle. It was traditionally made out of several disks that rotate moving in a unit. Once the numbers $9$ passed the disk returned to zero and the rightmost disk moves in a unit. The mathematical odometer version we will be interested in consists of infinitely many disks and has only two digits.  Endow the space $\{0,1\}$ with the discrete topology and $\Sigma_2= \{0,1\}^{\N}$ with the product topology.   The  {\em symbolic dyadic odometer} is the map $D: \Sigma_2 \to \Sigma_2$ defined by
\begin{equation*}
D((\underbrace{1,\dots, 1}_{k\textrm{-times}},0,w_{k+2},w_{k+3},\dots))=(\underbrace{0,\dots, 0}_{k\textrm{-times}},1,w_{k+2},w_{k+3},\dots),
\end{equation*}
and $D((1,1, \dots))=(0,0, \dots )$.  The map $D$ is also known as the von Neumann-Kakutani odometer, adding machine or $2-$odometer. It can also be thought of as the system induced by the addition of $\underline{1}=(1,0,0, \dots)$. This group theoretic point of view allows for several generalizations. From the dynamical systems point of view the odometer is  well understood. Among the many known properties of this system we will be interested in the following (see \cite[pp.25-26]{aa} and \cite[Section 4.5.2]{b}).

\begin{proposition} \label{erg_odo}
The symbolic dyadic odometer satisfies:
\begin{enumerate}
\item \label{1} For every $n \in \N$ and $w \in \Sigma_2$ we have,
\begin{equation*}
\left\{ (D^k w)_1, \dots , (D^k w)_n ): 0 \leq k \leq 2^n -1 \right\} = \{0,1\}^n.
\end{equation*}
 \item The map $D$ is minimal, that is, every orbit is dense.
 \item The map $D$ is uniquely ergodic and its unique invariant measure is the $(1/2, 1/2)$ Bernoulli measure.
\item The map $D$ has zero topological entropy.
\end{enumerate}
\end{proposition}

It readily follows from Proposition \ref{erg_odo} \eqref{1} that the orbit of $w=(0,0, \dots)$ with respect to $D$ induces  a bijection between the set of natural numbers and  the set of all sequences in $\Sigma_2$ ending with a tail of zeroes. Indeed,  it suffices to consider $n \mapsto D^n(0,0,\dots)$. This odometer allow us to \emph{count} sequences ending with a tail of zeroes.

The dyadic odometer can be realized in the unit interval. It should come as no surprise that this realization is related to the binary expansion of real numbers. Every $x \in [0,1)$ can be written in binary form as
\begin{equation*}
x= \sum_{i=1}^{\infty} \frac{w_i}{2^i},
\end{equation*}  
where $w_i \in \{0,1\}$.  Thus, to every $x \in [0,1)$  we can associate a sequence of  zeroes and ones,
\begin{equation*}
x=(w_1, w_2, w_3, \dots).
\end{equation*}
Let $\mathbb{D}= \{ x \in [0,1) : x= \frac{k}{2^n} \text{ for } k \in \N \cup\{0\} , n \in \N\}$ be the set of \emph{dyadic numbers} in $[0,1)$. Points in $[0,1) \setminus \mathbb{D}$ have a unique binary representation, while  those  in $\mathbb{D}$ have exactly two representations, one ending with a tail of zeroes and the other with a tail of ones. Unless stated {{}{to}} the contrary, we will always consider the binary expansion ending with a tail of  zeroes.  In view of this, there is a map
 $\pi_2:[0,1) \to \Sigma_{2}$ which assigns to each $x \in [0,1)$ the unique sequence $w \in \Sigma_{2}$  
 that corresponds to its binary expansion and does not end with a tail of ones. Moreover,  if we denote by
 $\Sigma_2^*= \Sigma_2 \setminus  \{w \in \Sigma_2 : w \text{ ends with a tail of } 1's \text{ or  of } 0's \} $ then the map
  \begin{equation*}
 \pi_2:[0,1) \setminus \mathbb{D} \mapsto \Sigma_2^*
  \end{equation*}
 is {{}{a}} homeomorphism.
 
\begin{definition}
Let  $(I_n)_n$ be the partition of the {{}{unit}} interval, where $I_n=[\frac{2^{n-1}-1}{2^{n-1}}, \frac{2^n-1}{2^{n}})$, for $n \in \N$. The 
interval realization of the \emph{dyadic odometer} is the interval map $D_2: [0,1) \to [0,1)$ defined by
\begin{equation*}
D_2(x) = x +\frac{3}{2^n}-1, 
\end{equation*}
for $x \in I_n$.
\end{definition}

\begin{proposition} \label{con_d}
The map $D: \Sigma_2^*\to \Sigma_2^*$ is topologically conjugated to $D_2: [0,1]\setminus \mathbb{D}  \to [0,1)   \setminus \mathbb{D}$ via the map $\pi_2$. That is, if $x \in [0,1]\setminus \mathbb{D}$ then
\begin{equation*}
(D \circ \pi_2)(x)=  ( \pi_2 \circ D_2)(x). 
\end{equation*}
\end{proposition}
\begin{proof}
Note that if $x \in I_n$ then its binary expansion  starts with $0.\underbrace{1\dots 1}_{n-1}0$. Also,
\[
D_2(x)=x-\left(1-\frac{1}{2^{n-1}}\right)+\frac{1}{2^n}=x-(0.\underbrace{1\dots 1}_{n-1})+(0.\underbrace{0\dots0}_{n-1}1).
\]
The result then follows.
\end{proof}

The map $D_2$ is an infinite interval exchange transformation. From Proposition \ref{con_d} and the properties of $D$ the following holds.

\begin{proposition} \label{erg_odometro}
The  dyadic odometer $D_2$ has zero topological entropy, it is minimal and uniquely ergodic. It preserves the Lebesgue measure of $[0,1)$.  
\end{proposition} 

\begin{remark} \label{orbita_diadica}
Note that the orbit of $x=0$ with respect to $D_2$ induces a bijection between the natural numbers and the dyadic numbers. 
\end{remark}

\section{Accelerated binary expansions and the Linear Renyi map.}

There is a natural way to relate to a dynamical system an \emph{accelerated} version of it. This speed up system is sometimes called {{}{a}}  \emph{jump transformation} and was studied in detail by  Fritz Schweiger \cite{sc}. Our starting system is the following map.
\begin{definition}
The \emph{Doubling map} is the function $B:[0,1] \to [0,1]$ defined by
\begin{eqnarray*}
B(x)=
\begin{cases}
2x & \text{ if } x \in [0,1/2) \\
2x-1& \text{ if } x \in [1/2, 1].
\end{cases}
\end{eqnarray*}
\end{definition}
This system is closely  related to the binary expansion. Indeed, the map $B$ acts as the shift in the binary expansion of a number in $[0,1]$.

Consider now the first hitting time map in  $[0,1/2)$. This function  defined away from the pre-images of $x=0$ by
\begin{equation*}
\tau_B(x):= 1 + \min \left\{n \in \N \cup\{0\} : 	B^n(x) \in \left[0, \frac{1}{2} \right)	\right\}.
\end{equation*}
%Denote by $(I_n)_n$ the partition defined by $I_n=\left( \frac{1}{2^{n+1}} , \frac{1}{2^{n}} \right]$. 
For every $x \in [0,1]$, not in the boundary  of the partition defined by $(I_n)_n$, we define the \emph{linear Renyi map} (or the jump transformation for $B$) by
\begin{equation*}
R_2(x)= B^{\tau_B(x)}(x).
\end{equation*}
The map $R_2$ is  piecewise linear and full-branched in the partition $(I_n)_n$. It has infinite topological entropy and preserves the Lebesgue measure. It  will be convenient to describe symbolically the action of $R_2$.  For every integer $k \geq 0$ we  define the block

\begin{eqnarray*}
b_{k}=
\begin{cases}
0 & \text{ if  } k=0,  \\
\underbrace{1 \dots  1}_{k\textrm{-times}} 0 & \text{ if } k \geq 1.
\end{cases}
\end{eqnarray*}

\begin{lemma}\label{lema_bk}
The binary expansion of every real number $x\in [0,1)$ can be written in a unique way as an infinite concatenation of blocks $b_k$'s, 
\[
x=0.b_{k_1}b_{k_2}b_{k_3}\dots 
\]
\end{lemma}

For example, 
\begin{equation*}
0. \underbrace{0}_{k_1=0}\underbrace{0}_{k_2=0}\underbrace{1110}_{k_3=3}\underbrace{10}_{k_4=1} \underbrace{111110}_{k_5=5} \underbrace{10}_{k_6=1} \dots =0.b_0b_0b_3b_1b_5b_1\dots
\end{equation*}

For $x\geq 1$ we define $[x]_2:=2^{[\log_2 x]}$.

\begin{lemma} \label{l21}
If $x=0.b_{k_1}b_{k_2}\dots \in [0,1)$ then
\[
\left[\frac{1}{1-x}\right]_2=2^{k_1}.
\]
\end{lemma}

\begin{proof}
For $x\in [0,1)$ consider its binary representation  $x=0.w_1w_2w_3\dots \ $. With the notation $\bar w=1-w$, we obtain that
\[
1-x=0.\bar w_{1}\bar w_2\bar w_3\dots 
\]
Define the blocks 
\[
\bar b_k=\underbrace{0\dots 0}_{k-\textrm{times}}1.
\]
We thus have $1-x=0.\bar b_{k_1}\bar b_{k_2}\bar b_{k_3}\dots $. Setting,  $r=0. b_{k_2} b_{k_3}\dots\in [0,1)$,  we obtain
\begin{eqnarray*}
1-x&=&\frac{1}{2^{k_1+1}}+\frac{1}{2^{k_1+1}}\left(0.\bar b_{k_2}\bar b_{k_3}\dots \right)=\frac{1}{2^{k_1+1}}(2-r).
\end{eqnarray*}
Since $1<2-r\leq 2$, we obtain
\[
2^{k_1}\leq \frac{1}{1-x}<2^{k_1+1}.
\]
\end{proof}
Reminiscent of formula (\ref{Renyi_form}), a formula for $R_2$ in terms of $[\cdot]_2$ can be obtained. Indeed, it follows from Lemmas \ref{lema_bk} and \ref{l21} that:

\begin{lemma}
{{}{For}} $x \in [0,1)$ the linear Renyi map can be written as,
\begin{equation*}
R_2(x)=2-2\frac{\left[\frac{1}{1-x}\right]_2}{\frac{1}{1-x}}.
\end{equation*} 
\end{lemma}

The map $R_2$ is conjugated to the shift on a countable alphabet.

\begin{lemma} \label{cms}
The map $\sigma:\Sigma_{\geq 0}\to \Sigma_{\geq 0}$ is topologically conjugated to $R_2:[0,1) \setminus \mathbb{D} \to 
[0,1) \setminus \mathbb{D} $.
\end{lemma}

\begin{proof}
Note that the map $\eta: \Sigma_2^* \to \Sigma_{\geq 0}$ defined by 
\begin{equation*}
\eta((w_1,w_2, \dots))= \eta (b_{k_1}, b_{k_2}, \dots)= (k_1, \ k_2,  \ \dots),
\end{equation*}
is a homeomorphism. The map $\eta \circ \pi_2$ is the required conjugacy.
\end{proof}

Both the Renyi map and the linear Renyi map are topologically conjugated to the full shift on a countable alphabet. Therefore,

\begin{proposition} \label{con}
There exists {{}{a}} homeomorphism $\phi: [0,1) \setminus \mathbb{D} \to [0,1) \setminus \Q$ such that
\begin{equation*}
\phi \circ R_2= R \circ \phi.
\end{equation*}
\end{proposition}

 \begin{proof}
 Consider the homeomorphism 
 $h:\Sigma_{\geq 0}  \to \Sigma_{\geq 2}$ defined by 
 \begin{equation*}
 h(a_1, a_2, a_3, \dots)=(a_1+2,a_2+2, a_3 +2, \dots ).
 \end{equation*} 
 By virtue of Lemmas \ref{code}  and \ref{cms}, the map $\phi= \pi_R \circ h \circ \eta \circ \pi_2$ satisfies the required properties. 
 \end{proof}

\subsection{The odometric action}

We conclude this section {{}{by}} describing the action of the dyadic odometer on the block codes of the accelerated system studied above.  Let $ \tilde{D}=\eta\circ D \circ \eta^{-1}$.
 
 \begin{lemma}  \label{key}
 The following holds
 \begin{equation}
\tilde{D}(k_1, k_2,  \dots)=(\underbrace{0,\dots,0,}_{k_1\textrm{-times}}(k_2+1), \dots).
 \end{equation}
  \end{lemma}
\begin{proof}
We have $\eta^{-1}(k_1, k_2, \dots)=(\underbrace{1,\dots, 1,}_{k_1\textrm{-times}}0, b_{k_2}, \dots)$, and thus 

\[
D\left(\underbrace{1,\dots, 1,}_{k_1\textrm{-times}}0, b_{k_2},\dots \right)=\left( \underbrace{0,\dots, 0}_{k_1\textrm{-times}},1,b_{k_2},\dots\right).
\]
The conclusion follows from the fact that $1, b_k=b_{k+1}$.
\end{proof}

\section{The map $D_2$ is a linear map in the conjugacy class of $T$.} \label{sec:conj}

As promised, in our next result we prove that the odometer $D_2$ is topologically conjugated to the map $T$.
Recall that  $\mathbb{D}$ denotes  the set of dyadic numbers in $[0,1)$.

\begin{theorem} \label{main}
The {{}{homeomorphism}} $\phi:[0,1) \setminus \mathbb{D} \to [0,1) \setminus \Q$, defined by $\phi = \pi_R \circ h \circ \eta\circ \pi_2$, is such  that 
\begin{equation}
\phi \circ D_2= T \circ \phi.
\end{equation}
\end{theorem}  
 
 \begin{proof}
 Let $O: \Sigma_{\geq 2} \to \Sigma_{\geq 2}$ be the \emph{odometric substitution} defined by
 \begin{equation*}
O(a_1, a_2, a_3, \dots)= (\underbrace{2,2,\dots,2}_{(a_1-2)\textrm{-times}}, a_2+ 1, a_3, a_4, \dots).
\end{equation*} 
We have that $O \circ h =h \circ \tilde{D}$. It then follows from Lemma \ref{key} and Theorem  \ref{action} that 
the {{}{homeomorphism}} $\phi= \pi_R \circ h \circ \eta\circ \pi_2$ has the desired properties.
\end{proof}

\begin{figure} 
  \centering
  \includegraphics[width=6cm]{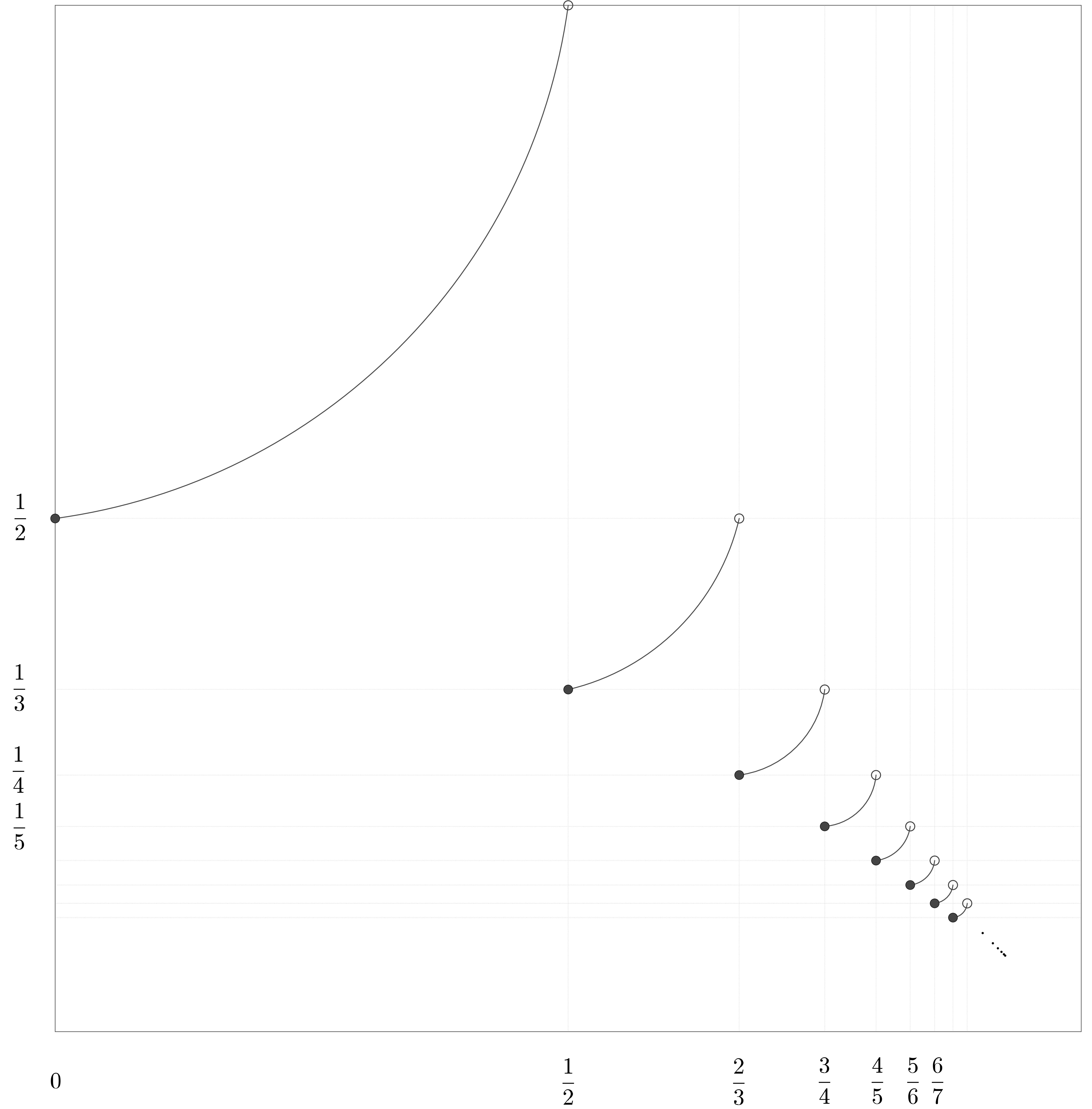}
 \includegraphics[width=6cm]{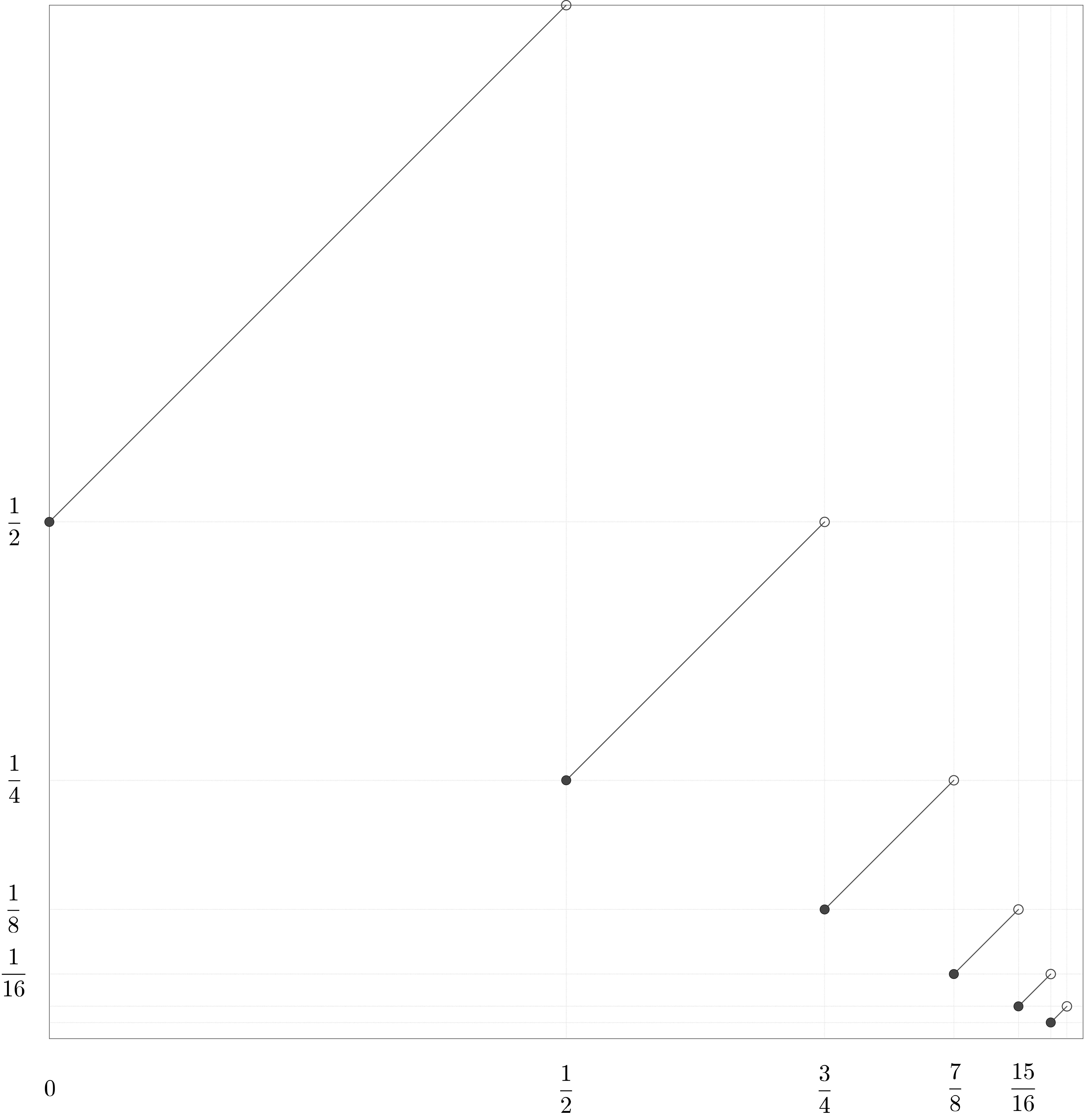}
  \caption{The map $T$ (left) and the realization of the dyadic odometer (right).}\label{fig}
  \end{figure}

%
%\begin{remark} It turns out that diagram contained in the above proof also holds (in a point set meaning) if we replace 
%\[
%[0,1)\setminus \mathbb{D},\  \Sigma_{2,0}\setminus \Sigma_{\bar 0},\  \Sigma_*\setminus\Sigma_{*, \bar 0},\  \Sigma_{\geq 2}\setminus \Sigma_{\bar 2},\  [0,1)\setminus \mathbb{Q}
%\]
%by
%\[
%[0,1),\  \Sigma_{2,0},\  \Sigma_*,\  \Sigma_{\geq 2},\  [0,1)
%\]
%respectively. Hence, the map $\phi$ also conjugates the action of $D_2$ over the dyadic numbers with the action of $T$ over the rational numbers.
%
%\end{remark}

\begin{remark} \label{orbit_rationals}
The map $\phi= \pi_R \circ h \circ \eta\circ \pi_2: \mathbb{D} \to [0,1] \cap \Q$ is a bijection satisfying $\phi \circ D_2= T \circ \phi$. 
\end{remark}

Theorem \ref{action} together with Remark \ref{orbit_rationals}  provides a new proof of the result by Calkin and Wilf \cite{cw} and Newman \cite{k} on the construction of explicit bijections between the natural numbers and the rationals in $[0,1)$. Indeed,

\begin{corollary}[Calkin and Wilf, Newman] \label{rationals}
The map $n \mapsto T^n(0)$ is a  bijection between $\N$ and $\Q \cap(0,1)$. Consequently, the map $n\mapsto F^n(0)$ is a  bijection between $\N$ and $\Q \cap \mathbb{R}^+$.
\end{corollary}

\begin{proof}
As pointed out in Remark \ref{orbita_diadica}, the orbit of $x=0$ with respect to $D_2$ induces a bijection of the natural numbers and the dyadic numbers. In view of Remark \ref{orbit_rationals} we have that {{}{the}} orbit of $x=0$ with respect to $T$ induces a bijection  between $\N$ and $\Q \cap (0,1)$. Indeed, this bijection is given by $n \mapsto T^n(0)$.
\end{proof}

%
%Recall that there exists a bijection between rational numbers in $[0,1)$ and the set $\mathcal{R}$ of sequences of positive integers larger or equal to $2$ ending with a tail of $2's$. The bijection is given by  $\pi_R$, the coding map of the backward continued fraction. The orbit of $0=(2,2,2, \dots)$ with respect to $T$ is nothing but the action of an odometer, therefore every  $x=(a_1,a_2, \dots, a_n, 2,2,\dots) \in [0,1)$ having a backward continued fraction representation in $\mathcal{R}$ is attained once and only once by the orbit of zero. {{} GI: quiz\'as m\'as explicaci\'on?}

We now recover results by Bonanno and Isola \cite{bi} regarding dynamical properties for the map $T$.

\begin{corollary}[Bonanno and Isola] \label{ergodic:T}
The map $T$ has zero entropy, it is minimal and uniquely ergodic. Moreover, its unique invariant measure is not absolutely continuous with respect to the Lebesgue measure.
\end{corollary} 

\begin{proof}
The fact that $T$ has zero entropy is minimal and uniquely ergodic when restricted to $[0,1)\setminus \Q$ is direct from Theorem \ref{main} and the properties of the dyadic odometer, Proposition \ref{erg_odo}. Moreover, as established in Corollary \ref{rationals}, the orbit of $x=0$ induces a bijection between $\N$ and $\Q \cap (0,1)$. In particular, the orbit of any rational number by $T$ is dense and no invariant measure is supported on it. Hence, $T$  is minimal, uniquely ergodic and has zero topological entropy.

Recall now that the Lebesgue measure on $[0,1)$, that we denote by $\text{Leb}$, is invariant for the linear Renyi map $R_2$ (being full branched and piecewise linear). Moreover, since the systems $R_2$ and the Renyi map $R$ are topologically conjugated by $\phi$ (see 
Proposition \ref{con}), the push-forward of $\text{Leb}$ with $\phi$ is an invariant probability measure for $R$. Since $R$ does not have  invariant probability measures absolutely continuous with respect to the Lebesgue measure, we have that this push-forwarded measure is not absolutely continuous with respect to the Lebesgue measure. 

Note now that the unique invariant probability for $T$ is the push-forward by the same map $\phi$ of the unique invariant probability for $D_2$. The result now follows since $\text{Leb}$ is the unique invariant measure for $D_2$.
\end{proof}

\begin{remark}
 As interval exchange transformations both the map $T$ and the dyadic odometer share the same permutation but have different partitions. We stress that even for finite non{{}{-}}linear interval exchange transformations it is a very difficult problem to find a linear interval exchange transformation in its conjugacy class. Work in this direction has been an active area of research, see for example \cite{mmy}. This problem is, of course,  motivated by the corresponding question for circle homeomorphism and  rotations.
 \end{remark}

\section{Minkowski's Question Mark function.}

Continuous, non-decreasing functions mapping the unit interval onto itself and yet singular, that is, having derivative equal to zero Lebesgue almost everywhere, were constructed already in 1884 by Georg Cantor \cite{c}.  In 1904, Hermann Minkowski \cite{m} introduced a function of the same type but with the remarkable property of being  \emph{strictly increasing}. The function $?:[0,1] \to [0,1]$, called \emph{Question Mark function} can be defined inductively on the rational numbers in $[0,1]$  by $?(0)=0$, $?(1)=1$ and
\begin{equation*}
? \left( \frac{p+p'}{q+q'}	\right)= \frac{?(p/q)+ ?(p'/q')}{2}, 
\end{equation*}
whenever $p/q$ and $p'/q'$ are rational numbers in lowest terms with $p'q-pq'=1$. The definition for irrationals in $[0,1]$ follows by continuity. Salem \cite{s}  proved that the question mark function is actually H\"older and conjectured that its Fourier-Stieltjes  coefficients decay to zero at infinity with a polynomial rate. This was recently established by Thomas Jordan and Tuomas Sahlsten \cite{js}. Denjoy \cite{d1,d2}, who actually proved the singularity of the function $?$, also obtained an explicit formula for it. In order to explain this, we briefly recall the definition and basic properties of the classical  \emph{continued fraction} expansion. Every real number $ x \in (0,1)$  can be written as a continued fraction of the form
\begin{equation*}
x = \textrm{ } \cfrac{1}{a_1 + \cfrac{1}{a_2 + \cfrac{1}{a_3 + \dots}}} = \textrm{ } [a_1, a_2, a_3, \dots]_C,
\end{equation*}
where $a_i \in \mathbb{N}$. If $x \in (0,1)$ is irrational then the continued fraction expansion is unique and it has infinitely many terms \cite[Theorem 170]{hw}. Rational numbers have two different expansions, both of them finite and one of them having last digit equal to $1$. As is the case of the backward continued fraction, there exits a dynamical system closely related to the continued fraction expansion. The \emph{Gauss} map  $G :(0,1] \to (0,1]$, is the interval map defined by 
\begin{equation*}
G(x)= \frac{1}{x} -\left[ \frac{1}{x} \right].
\end{equation*}
The Gauss map acts as the shift map on the continued fraction expansion,
\begin{equation*}
 a_n = \left[\frac{1}{G^{n-1}(x)} \right].
\end{equation*} 
We refer to  Khinchine's  well known 1935 monograph on continued fractions \cite{kh} for further properties. We can now state Denjoy's result.

\begin{proposition}[Denjoy]\label{denjoy}
If $x=[a_1, a_2, a_3, \dots]_C$ is an irrational number in $[0,1]$ then
\begin{equation*}
?(x)= \frac{1}{2^{a_1-1}} - \frac{1}{2^{a_1+a_2-1}} + \frac{1}{2^{a_1+a_2+a_3-1}}- \dots = 2 \sum_{n=1}^{\infty} \frac{(-1)^{n+1}}{2^{a_1+a_2 + \dots + a_n}}.
\end{equation*}
\end{proposition}
 
 The following functional properties of the question mark function hold.

\begin{lemma} \label{pro}
For every $x \in [0,1]$ we have that
\begin{eqnarray*}
?\left( \frac{x}{x+1} \right) = \frac{?(x)}{2} \quad \text{ and }  \quad 
?(1-x)= 1- ?(x).
\end{eqnarray*}
\end{lemma}

A remarkable property of the Question Mark function,  and one of the reasons for Minkowski to define it, is that it gives a bijection between quadratic irrationals of the unit interval and rationals in the same interval. Moreover, it provides a bijection between rational and dyadic numbers. This later property is, of course, of interest to us since we have studied a bijection between the natural numbers and the rational ones that passes through the dyadic numbers. The relation with the Question Mark function is not, as we will see below, {{}{a}} coincidence. In order to explain this we require the following result that relates the backward with the continued fraction expansions. This corresponds to \cite[Exercise 10. p.128]{ka}.

 \begin{lemma}\label{CFaBCF}
If $x\in [0,1)$ is an irrational  number with infinite continued fraction expansion $x=[m_1, m_2, m_3, \dots]_C$. Then,  the backward continued fraction of $x$ is
\begin{equation}\label{deCFaBCF}
x=[\ \underbrace{2,\dots, 2}_{m_1-1}, \ (m_2+2),\  \underbrace{2,\dots, 2}_{m_3-1}, \ (m_4+2),\  \dots].
\end{equation}
\end{lemma}

\begin{proof} Denote by $G^2(x)_{BCF}$ the backward continued fraction of the second iterate of the  Gauss map, $G^2(x)$. It suffices to show that the backward continued fraction of $x$ verifies
\[
x=[\ \underbrace{2,\dots, 2}_{m_1-1},\ (m_2+2),\ G^2(x)_{BCF}].
\]
We will prove this inductively on $m_1$. The following direct, but useful, identity will be crucial in our reasoning,
\[
\frac{1}{1+\frac{1}{t}}=1-\frac{1}{t+1}.
\]
Assume $m_1=1$. Then
\begin{eqnarray*}
[1, m_2,m_3, \dots]_C&=&\frac{1}{1+\frac{1}{m_2+G^2(x)}}=1-\frac{1}{m_2+1+G^2(x)}\\
&=&1-\frac{1}{m_2+2 +G^2(x)-1}
=[(m_2+2)\ G^2(x)_{BCF}].
\end{eqnarray*}
Assume that (\ref{deCFaBCF}) holds for $m_1\geq 1$. Then, 

\begin{eqnarray*}
 [(m_1+1),\ m_2, m_3,\dots]_C  &=& \frac{1}{1+m_1+\frac{1}{m_2+G^2(x)}}=\frac{1}{1+\frac{1}{\left(m_1+ \frac{1}{m_2+G^2(x)}\right)^{-1}}}   \\
&=&  1- \frac{1}{1+\left(m_1+\frac{1}{m_2+G^2(x)}\right)^{-1}}  \\
&=& 1-\frac{1}{2+\frac{1}{\left(m_1+\frac{1}{m_2 +G^2(x)}\right)}-1}  =  [2, \ [m_1, m_2, \dots]].
\end{eqnarray*}
\end{proof}

As seen in Proposition \ref{denjoy}, Denjoy obtained an analytical formula for the question mark function associated to the continued fraction expansion. We now obtain the corresponding formula for the backward continued fraction. 

\begin{proposition} \label{?-b}
If $x\in [0,1)$ has backward continued fraction expansion $x=[a_1, a_2, \dots]$ then
\begin{eqnarray*}
?(x)&=&1-\sum_{j=1}^{\infty}\frac{1}{2^{-j+a_1+a_2\dots+a_j}}, \\
&=&0.b_{a_1-2}b_{a_2-2}\dots
\end{eqnarray*}
\end{proposition}

\begin{proof} For  $x= [a_1, a_2, \dots ]$ let
\[
\hat ?[a_1, a_2, \dots]=0.b_{a_1-2}b_{a_2-2}\dots . 
\]
Let $x=[m_1, m_2, \dots]_{C}$ be an irrational real number $x\in [0,1)$. By Lemma \ref{CFaBCF}   the backward continued fraction of $x$ equals 
\[
x=[\underbrace{2,\dots, 2}_{m_1-1}, \ (m_2+2),\  \underbrace{2,\dots, 2}_{m_3-1},\  (m_4+2), \  \dots].
\]
Hence
\begin{eqnarray*}
\hat ?(x)&=&0.\underbrace{b_0\dots b_0}_{m_1-1} b_{m_2} \underbrace{b_0\dots b_0}_{m_3-1}b_{m_4}\underbrace{b_0\dots b_0}_{m_5-1}\dots\\
&=&0.\underbrace{0\dots 0}_{m_1-1} \underbrace{1\dots 1}_{m_2}0 \underbrace{0\dots 0}_{m_3-1}\underbrace{1\dots 1}_{m_4}0 \underbrace{0\dots 0}_{m_5-1}\dots\\
&=&\underbrace{0.0\dots 0}_{m_1} \underbrace{1\dots 1}_{m_2} \underbrace{0\dots 0}_{m_3}\underbrace{1\dots 1}_{m_4} \underbrace{0\dots 0}_{m_5}\dots .
\end{eqnarray*}
The above expression corresponds to the classic Question Mark function $?$ for a real number written in continued fraction expansion, and hence $\hat ?=?$ on irrationals. The conclusion follows from continuity.  
\end{proof}
The above discussion allows us to conclude that the conjugacy constructed in Proposition \ref{con} and Theorem \ref{main} is nothing but the Question Mark function. Therefore, we recover the following result by Bonanno and Isola \cite[Theorem 2.3]{bi},
\begin{proposition}
The maps $T$ and $D_2$ are topologically conjugated via the Minkowski Question Mark function.
\end{proposition}

\section{Backward Farey Map.} \label{epi}
As indicated  in the Introduction,  the backward continued fraction expansion is naturally related with the binary expansion. This relation emerges from the following scheme,
\begin{equation*}
\text{Renyi map} \xleftrightarrow{accelerated} \text{Backward Farey map} \xleftrightarrow{?} \text{Doubling map}
\end{equation*}
In this section we provide details to explain it.

\begin{definition}
The \emph{Backward Farey map} is the function $J:[0,1] \to [0,1]$ defined by
\begin{eqnarray*}
J(x)=
\begin{cases}
\frac{x}{1-x} & \text{ if } x \in [0,1/2) \\
\frac{2x-1}{x}  & \text{ if } x \in [1/2, 1].
\end{cases}
\end{eqnarray*}
\end{definition}

\begin{remark}
This map has two \emph{increasing} branches. The derivatives at $x=0$ and at $x=1$ are equal to $1$. If $G$ denotes the Gauss map, then for $x \in [1/2, 1]$ we have that $J(x)= 1- G(x)$. With a different purpose, this map was also studied by Bonanno and Isola \cite{bi}.
\end{remark}
%
%\begin{definition}
%The \emph{Doubling map} is the function $B:[0,1] \to [0,1]$ defined by
%\begin{eqnarray*}
%B(x)=
%\begin{cases}
%2x & \text{ if } x \in [0,1/2) \\
%2x-1& \text{ if } x \in [1/2, 1].
%\end{cases}
%\end{eqnarray*}
%\end{definition}
%The map $B$ acts as the shift in the binary expansion of a number in $[0,1]$.

\begin{proposition}
The Backward Farey map is topologically conjugated to the doubling map via {{}{the}} Question Mark function. That is, for every $x \in [0,1]$ we have
\begin{equation*}
(B \circ ?)(x)=(? \circ J)(x).
\end{equation*}
\end{proposition}

\begin{proof}
Let us consider first the case in which $x \in [0, 1/2)$. Note that the function $\frac{x}{1-x}$ is the inverse of $\frac{x}{x+1}$. Thus, by Lemma \ref{pro} we obtain,
\begin{eqnarray*}
?(x)= ?\left(\frac{\frac{x}{1-x}}{\frac{x}{1-x} +1}	\right)= \frac{1}{2} ? \left(\frac{x}{1-x}	\right).
\end{eqnarray*}
Therefore,
\begin{equation*}
(B \circ ?)(x)= 2?(x)= ? \left(\frac{x}{1-x}	\right)= (? \circ J)(x).
\end{equation*}
 Assume now that $x \in [1/2, 1 ]$. In this case the continued fraction expansion of the point has the form $x=[1, a_2, a_3, \dots]_C$. By Dejoy's formula for $?$ we have,
 \begin{eqnarray} \label{cuenta}
 ?(x)= 2 \sum_{n=1}^{\infty} \frac{(-1)^{n+1}}{2^{1+a_2+a_3+ \dots a_n}}= 2\left(\frac{1}{2} - \frac{1}{2^{1+a_1}} +\frac{1}{2^{1+a_1+a_2}} - \dots		\right) =\\
  2\left(\frac{1}{2} \left( 1 - \frac{1}{2^{a_1}} +\frac{1}{2^{a_1+a_2}} - \dots	\right)	\right) =
  1 + \left(	 - \frac{1}{2^{a_1}} +\frac{1}{2^{a_1+a_2}} - \dots		\right).
 \end{eqnarray}
From \eqref{cuenta} we  have that:
\begin{equation*}
2?(x)= 2 - ?(G(x)).
\end{equation*}
In particular,
\begin{equation} \label{un-lado}
2?(x)-1= 1 -?(G(x)).
\end{equation}
Note now that, for $x \in [1/2, 1]$ we have that $J(x)= 1- G(x)$. Therefore, by Lemma \ref{pro},
\begin{equation} \label{otro-lado}
(? \circ J)(x)=?(1-G(x))= 1- ?(G(x)).
\end{equation}
Combining equations \eqref{un-lado} and \eqref{otro-lado} we have that if $x \in [1/2, 1]$ then
\begin{equation*}
(B \circ ?)(x)=(? \circ J)(x).
\end{equation*}
\end{proof}

%\subsection{The Renyi map is the accelerated version of the Backward Farey map}

It is well known, see for example \cite[Section 2.1]{i2}, that the Gauss map is an accelerated version of the Farey map. We now  prove an analogous result for the Renyi map.  We first consider the following partition of the unit interval. For every $n \in \N$, let
\begin{equation*}
P_n= \left[	\frac{n-1}{n}, \frac{n}{n+1}	\right).
\end{equation*}

\begin{remark}
For every $n \geq 1$ we have that $J(P_{n+1})= P_n$. Indeed,
\begin{equation*}
J\left( \frac{n}{n+1}\right) = 1 - G\left( \frac{n}{n+1}\right) =
1 - \left( \frac{n+1}{n} -\left[ \frac{n+1}{n}\right] \right)  = \frac{n-1}{n}.
\end{equation*} 
\end{remark}
 
Consider now the first hitting time to $P_1$. This function is well defined away from the pre-images of $1$ and is given by,
\begin{equation*}
\tau(x):= 1 + \min \left\{n \in \N \cup\{0\} : 	J^n(x) \in P_1	\right\}.
\end{equation*}
Note that if $x \in P_n$ then $\tau(x)= n$, hence
\begin{equation*}
\tau(x)= \left[\frac{1}{1-x}	\right].
\end{equation*}
We have:

\begin{lemma}
For every $x \in [0,1]$, not in the boundary  of the partition defined by $(P_n)_n$, we have
\begin{equation*}
R(x)= J^{\tau(x)}(x).
\end{equation*}
\end{lemma}

\begin{proof}
If $x \in P_n$ then $\tau(x)=n$. Hence $J^{\tau(x)}(x)= J^n(x)=J \circ J^{n-1}(x)$.
Recall that the Renyi map has the following form:
\begin{equation*}
R|_{P_n}(x)= \frac{nx- (n-1)}{1-x}.
\end{equation*}
On the other hand, if we denote the second branch of the backward Farey map by $\tilde{G}(x)=\frac{2x-1}{x}$, then
\begin{equation*} 
\tilde{G}^n(x)= \frac{(n+1)x -n}{nx-(n-1)}.
\end{equation*}
Thus, since the Renyi map and the map $J$ coincide in $[0, 1/2]$ we have
\begin{eqnarray*}
J^{\tau(x)}(x)= J^n(x)=J \circ J^{n-1}(x)= R\left( \frac{nx -(n-1)}{(n-1)x-(n-2)} \right) &=& \\ \frac{1}{1- \frac{nx -(n-1)}{(n-1)x-(n-2)}	} -1= \frac{nx -(n-1)}{1-x}.
\end{eqnarray*}
\end{proof}

%{{} GI: hay que ajustar n con n-1...} 

\end{document}